\RequirePackage{etoolbox}
\csdef{input@path}{
{style/}
{graphics/}
}
\documentclass[ejs,preprint,twoside,numbers,natbib]{imsart}
\RequirePackage[colorlinks,citecolor=blue,urlcolor=blue]{hyperref}
\usepackage{amsmath,amsthm,amsfonts,amssymb}
\usepackage{graphicx}
\usepackage{multirow}
\usepackage{enumerate}

\doi{10.1214/15-EJS1031}
\pubyear{2015}
\volume{9}
\issue{0}
\firstpage{1205}
\lastpage{1229}

\allowdisplaybreaks

\startlocaldefs

\newcommand{\vertiii}[1]{{\left\vert\kern-0.25ex\left\vert\kern-0.25ex\left\vert #1
\right\vert\kern-0.25ex\right\vert\kern-0.25ex\right\vert}}

\newtheorem{theorem}{Theorem}

\newtheorem{lema}{Lemma}
\newtheorem{assump}{Assumption}

\theoremstyle{definition}
\newtheorem{condition}{Condition}
\newtheorem{remark}{Remark}
\newtheorem{definition}{Definition}
\newtheorem{example}{Example}

\endlocaldefs

\begin{document}

\begin{frontmatter}
\title{Confidence intervals for high-dimensional inverse covariance estimation}
\runtitle{Confidence intervals for high-dimensional inverse covariance
estimation}

\begin{aug}
%
\author{\fnms{Jana} \snm{Jankov\'a}\corref{}\ead[label=e1]{jankova@stat.math.ethz.ch}}
\and
\author{\fnms{Sara} \snm{van de Geer}\ead[label=e2]{geer@stat.math.ethz.ch}}

\address{Seminar for Statistics\\
ETH Z\"urich\\
\printead{e1,e2}}
\end{aug}

\runauthor{J. Jankov\'a and S. van de Geer}

\begin{abstract}
We propose methodology for statistical inference for low-dimen\-sional parameters of sparse precision matrices in a high-dimen\-sional setting.
Our method leads to a non-sparse estimator of the precision matrix
whose entries have a Gaussian limiting distribution. 
Asymptotic properties of the novel estimator are analyzed  for the case of sub-Gaussian observations under a sparsity assumption on the entries of the true precision matrix and regularity conditions.
Thresholding the de-sparsified estimator gives guarantees for edge selection in the associated graphical model.
 Performance of the proposed method is illustrated in a simulation study.
\end{abstract}

\begin{keyword}[class=MSC]
\kwd[Primary ]{62J07}
\kwd[; secondary ]{62F12}
\end{keyword}

\begin{keyword}
\kwd{Confidence intervals}
\kwd{graphical Lasso}
\kwd{high-di\-mensional}
\kwd{precision matrix}
\kwd{sparsity}
\end{keyword}
%
%
\received{\smonth{3} \syear{2014}}

\end{frontmatter}
\maketitle

\section{Introduction}
A large number of methods has been proposed for the problem of
inverse covariance estimation in high-dimensional settings, where the
number of
parameters may be much larger than the sample size.
Common procedures in literature typically take advantage of
thresholding which leads to estimators whose
asymptotic distribution largely depends on the underlying unknown
parameter \cite{knight2000} and
is in general not tractable, which makes it challenging to establish
any results for statistical inference.
In this paper, motivated by the semi-parametric approach adopted in
\cite{vdgeer13} and \cite{zhang},
we propose an asymptotically normal non-sparse estimator of the
precision matrix which leads to confidence regions and testing for
low-dimensional parameters.

The problem of estimating the inverse covariance matrix in high
dimensions naturally arises in a wide variety of application domains,
such as graphical modeling of brain connectivity based on FMRI brain
analysis \cite{fmri}, gene regulatory network discovery \cite{genes},
financial data processing, social network analysis and climate data analysis.
The development of methodology for high-dimensional inference is of
interest for instance in differential networks, which comprise two
sample comparisons of high-dimensional graphical models where the goal
is to test equality of networks corresponding to two different
populations. Differential networks find application e.g. in cancer studies
\cite{mukherjee}.

Consider an i.i.d. sample $X_1,\dots,X_n \in\mathbb R^p$ of size $n$
from a zero-mean distribution with unknown covariance matrix $\Sigma
^*\in\mathbb R^{p\times p}.$
Denoting the inverse covariance matrix, often referred to as the
precision or concentration matrix, as $\Theta^* = (\Sigma^*)^{-1},$
the goal is to the estimate $\Theta^*$ in a setting where $p\gg n.$
The most natural candidate for an estimator of the covariance
matrix is presumably the sample covariance matrix. However, when $p>n$,
the sample covariance matrix is singular with probability one.
Even when $p/n$ tends to a constant, the covariance matrix exhibits
poor performance \cite{johnstone2}.

Different structural assumptions have been imposed on the model to
allow for consistent estimation in the regime $p\gg n$, here we
consider in particular sparsity assumptions on the number of non-zero
elements of the precision matrix.
Let $\mathcal V:=\{1,\dots,p\}$, let $S\equiv S(\Theta^*):= \{(i,j)\in
\mathcal V\times\mathcal V:\Theta^*_{ij}\not= 0\}$
be the set of all non-zero entries of $\Theta^*$ and denote the
cardinality of $S$ by $s.$
Use $S^c(\Theta^*)$ for the complement of $S(\Theta^*)$ in $\mathcal
V\times\mathcal V.$
We shall impose a sparsity assumption on the maximum row cardinality of
$\Theta^*$; therefore define $d= d_n$ as follows
\[
d := \max_{i\in\{1,\dots,p\}} |\{j\in\mathcal V: \Theta^*_{ij}\not=0\}|.
\]
Estimation of precision matrices is of interest in Gaussian graphical
modeling where the entries of the precision matrix
represent conditional dependences between the variables \cite{lauritzen}.
Suppose that $X =(X^1,\dots,$ $X^p) \sim\mathcal N(0,\Sigma^*)$
and associate the variables $X^1,\dots,X^p$ with the vertex set
$\mathcal V=\{1,\dots,p\}$ of an undirected graph $\mathcal G=(\mathcal
V,\mathcal E)$ with an edge set $\mathcal E.$
A pair $(i,j)$ is included in the edge set if and only if the variables
$X^i$ and $X^j$
are not independent given all remaining variables. Under $X\sim
\mathcal N(0,\Sigma^*)$, a pair of variables is conditionally
independent given all remaining variables if and only if the
corresponding entry in the precision matrix $\Theta^*=(\Sigma^*)^{-1}$
is zero. Hence a pair of variables is contained in the edge set if and
only if the corresponding entry in the inverse covariance matrix $\Theta
^*=(\Sigma^*)^{-1}$ is non-zero (\cite{lauritzen}). Elements of the
precision matrix may thus be interpreted as the edge weights in the
Gaussian graphical model.
The parameter $d$ corresponds to the maximum node degree in the
associated Gaussian graphical model and thus sparsity assumptions on
$d$ translate to sparsity of the edges in the graphical model.

\subsection{Overview of related work}
Existing work on statistical inference in high dimensional settings
has mostly focused on inference for parameters in linear models and
generalized linear models \cite
{vdgeer13,zhang,meinshausen,stanford1,bootstrap1,bootstrap2,vdgeer12}.
In particular we mention the paper \cite{zhang} where a semi-parametric
projection approach was proposed for testing and construction of
confidence intervals for low-dimensional parameters.
The proposed method is based on the Lasso estimator for which
the\vadjust{\goodbreak}
Karush-Kuhn-Tucker conditions are ``inverted'' to obtain a
de-sparsified estimator. The approach leads to asymptotically normal
and efficient (in a semi-parametric sense) estimation of the regression
coefficients and an extension of the method to generalized linear
models is given in \cite{vdgeer13}.
The key assumption which allows for asymptotically normal estimation
requires sparsity of order $\sqrt{n}/\log p$ in the high-dimensional
parameter vector and the method relies on $\ell_1$ norm error bound of
the Lasso.
The paper \cite{stanford1} essentially follows the same approach as
\cite{zhang} but uses a different approach to find an approximate
inverse for the sample covariance matrix.

Further methodology for inference for the regression coefficients in
high-dimen\-sional regression includes methods based on sample
splitting \cite{pvals,wasserman}, bootstrapping approach \cite
{bootstrap1,bootstrap2}, inference after variable selection \cite{buja}
and other \cite{meinshausen,belloni1}.

Estimation of precision matrices is a problem closely related to linear regression
and in high dimensions has been extensively studied in terms of point
estimation. Less work has yet been done on inference for precision
matrices in this setting. We mention the work \cite{zhou}
which suggests a regression approach leading to an asymptotically
normal estimator for elements of the precision matrix, under row
sparsity of order $\sqrt{n}/\log p$, bounded spectrum of the true
precision matrix and Gaussianity of the underlying distribution.
The procedure regresses each pair of variables $(X^i,X^j)$ on all the
remaining variables for each $(i,j)\in\mathcal V\times\mathcal V$ to
obtain an estimate of the noise level of the conditional distribution
of $(X^i,X^j)$. This requires $\mathcal O(p^2)$ high-dimensional
regressions with the square-root Lasso (\cite{sqrtlasso}).

The large amount of work that has studied methodology for point
estimation of precision matrices (a selected list includes \cite{glasso,buhlmann,yuan2,cai,sunzhang,bickel2})
typically uses regularization in terms of $\ell_1$ norm or some sort of
thresholding of the sample covariance matrix.
Hence they do not immediately lead to results for inference, but we
show they may serve as good initial estimators to construct
asymptotically normal estimators.

Here we consider in particular the graphical Lasso, which minimizes the
negative Gaussian log-likelihood with regularization in terms of the
$\ell_1$ norm of the off-diagonal entries of the precision matrix and
has been studied in detail in several papers \cite{glasso,rothman,ravikumar} and \cite{yuan}.
The optimization problem corresponding to graphical Lasso is a convex
optimization problem that can be solved with coordinate descent methods
\cite{aspremont,glasso} in polynomial time.

{The asymptotic behaviour of the graphical Lasso has been studied in
\cite{rothman} (see also \cite{wainwright1}) which derives rates of
convergence in Frobenius norm of order $O((p+d)\log p /n)$ under mild
conditions on the eigenvalues of $\Theta^*$ and under sparsity
$(p+d)\log p /n\rightarrow0.$
High-dimensionality here is reflected in $p$ being allowed to grow as a
function of $n$, however, in limit, $p/n\rightarrow0$ is required.
The high-dimensional setting $p\gg n$ is considered in \cite
{ravikumar}, where convergence rates for the supremum norm of order
$O(\kappa_{\Gamma^*}\sqrt{\log p/n})$ are derived under an
irrepresentability condition on the true precision matrix $\Theta^*$,
sparsity $d^2\log p /n \rightarrow0$ and sub-Gaussian tails of the
underlying distribution. The rates depend on certain quantities $\kappa
_{\Gamma^*}$ and $\kappa_{\Sigma^*}$, where $\kappa_{\Gamma^*}$ is the
$\ell_1$ matrix norm of the inverse of a certain subset of the Hessian
matrix $\Gamma^* = \Sigma^*\otimes\Sigma^*$ and $\kappa_{\Sigma^*} $
the $\ell_1$ matrix norm of the true covariance matrix $\Sigma^*$. For
reader's convenience we discuss the results in more detail in Section
\ref{subsec:setup} and appendix \ref{subsec:tails}.

Further methodology on estimation of precision matrices in particular
includes the regression approach \cite{buhlmann,yuan2,cai}
and \cite{sunzhang} which uses a Lasso-type algorithm or Dantzig
selector \cite{candes2007} to estimate each column or a smaller part of
the precision matrix individually, thresholding of the sample
covariance matrix \cite{bickel2} or a combination thereof.

\subsection{Outline}
In this paper, we propose a de-sparsified estimator based on the
graphical Lasso
and study its theoretical properties for low-dimensional statistical
inference and edge selection in the associated graphical model.
The work closely follows the approach of \cite{vdgeer13}, which builds
on ``inverting'' the necessary Karush-Kuhn-Tucker conditions for an
optimization problem. The paper \cite{vdgeer13} demonstrates this
method for the case of linear regression and generalized linear models,
while we apply the idea to a fully nonlinear estimator. By inverting
the KKT conditions, we obtain a de-sparsified graphical Lasso estimator
and consequently we analyze its asymptotic properties. Asymptotic
normality of the new estimator is proved for sub-Gaussian observations,
under regularity conditions on the true precision matrix $\Theta^*$.
The estimator may be thresholded again to give guarantees for edge
selection in the associated graphical model. The performance of the
method is illustrated on both simulated and real data.

The paper is organized as follows. In Section \ref{subsec:setup}, we
briefly introduce the model.
Section \ref{sec:main} contains the main results. Section \ref{sec:emp}
illustrates the theoretical results in a simulation study and on a real
data set. Finally, Section \ref{sec:proofs} contains proofs.

\smallskip\noindent
\textit{Notation.} For two matrices $A$ and $B$, use $A\otimes B$ to
denote the Kronecker product of $A$ and $B.$
For a vector $x\in\mathbb R^d$ and $p\in(0,\infty]$ we use the
notation $\|x\|_p$ to denote the $p-$norm of $x$ in the classical
sense. For a matrix $A\in\mathbb R^{d\times d}$ we use the notations
$\vertiii{A}_\infty=\max_{i} \|e_i^T A\|_1$, $\vertiii{A}_1= \vertiii
{A^T}_\infty$ and $\|A\|_\infty=\max_{i,j}|A_{ij}|$.
The symbol $\text{vec}(A)$ denotes the vectorized version of a matrix
$A$ obtained by stacking rows of $A$ on each other.
By $e_i$ we denote a $p$-dimensional vector of zeros with one at
position $i$ and by $e_{ij}:=e_i\otimes e_j$ a $p^2$-dimensional vector
of zeros with one at position indexed by $(i,j).$

For sequences $f_n,g_n$, we write $f_n=O(g_n)$ if $|f_n| \leq C |g_n|$
for some $C>0$ independent of $n$ and all $n>C.$ Analogously, we write
$f_n=\Omega(g_n)$ if $|f_n| \geq C |g_n|$ for some $C>0$ independent of
$n$ and all $n>C.$ We write
$f_n\asymp g_n$ if both $f_n = \mathcal O(g_n)$ and $f_n =\Omega(g_n)$ hold.
Finally, $f_n=o(g_n)$ if $\lim_{n\rightarrow\infty} f_n/g_n =0.$

We use $S_{+}^p$ to denote the cone of positive semi-definite $p\times
p$ matrices, i.e.
$S_{+}^p:=\{A\in\mathbb R^{p\times p}| A=A^T, A \succeq0\}$ and
$S^{p}_{++}$ to denote the set of positive definite $p\times p$
matrices, $S_{++}^p:=\{A\in\mathbb R^{p\times p}| A=A^T, A \succ0\}$.

The components of a vector $X\in\mathbb R^p$ will be denoted by upper
indices, i.e. $X = (X^1,\dots,X^p)$. Elements of matrices will be
typically denoted by lower indices, e.g. $A_{ij}.$
We use $\rightsquigarrow$ to denote convergence in
distribution.\vadjust{\goodbreak}

\subsection{Model setup}\label{subsec:setup}

\begin{definition}
A real zero-mean random variable $X$ is sub-Gaussian if there exists
$K>0$ such that
%
\begin{equation}\label{defsg}
\mathbb Ee^{X^2/K^2} \leq2.
\end{equation}
\end{definition}

Condition \ref{subg} implies a bound on the moment generating function
$\mathbb E^{tX} \leq e^{\frac{3}{2}K^2t^2}$ for all $t>0$
and a tail bound $\mathbb P(|X|>t) \leq2e^{-\frac{t}{6K^2}}$ for all
$t>0,$ which are both equivalent characterizations of sub-Gaussianity.
A prime example of a sub-Gaussian random variable is a zero-mean
Gaussian random variable.

We shall consider the following sub-Gaussianity conditions for random
vectors $X=(X^1,\dots,X^p)$ with zero mean and covariance matrix $\Sigma^*$.
%
\begin{condition}[Sub-Gaussianity condition]\label{subg}
All normalized components
$X^i/\break \sqrt{\Sigma_{ii}^*},i=1,\dots,p$ of the zero-mean random vector
$X=(X^1,\dots,X^p)$ with covariance matrix $\Sigma^*$ are sub-Gaussian
random variables with a common parameter $K>0.$
\end{condition}

The condition \ref{subg} is weaker than requiring the sub-Gaussianity
of the whole vector $X=(X^1,\dots,X^p)$ in the following sense.
\begin{condition}[Sub-Gaussianity vector condition]\label{subgv}
A zero-mean random vector $X\in
\mathbb R^p$ satisfies the sub-Gaussianity vector condition if
there exists a constant $K>0$ such that
%
\begin{equation}
\sup_{\alpha\in\mathbb R^p:\|\alpha\|_2\leq1}\mathbb Ee^{|\alpha^T X
|^2/K^2} \leq2.
\end{equation}
\end{condition}

If a random vector $X=(X^1,\dots,X^p)$ satisfies \ref{subgv} with a
constant $K$, then each component $X^i$
satisfies \eqref{defsg} with $K$.

\smallskip
We now review some notation and results related to the graphical Lasso
estimator \cite{glasso} on which our further analysis is based.
Consider an i.i.d. sample $X_1,\dots,X_n$ distributed as $X$ with
$\mathbb EX=0,\text{cov}(X)=\Sigma^*.$
Let $\hat\Sigma= \frac{1}{n}\sum_{i=1}^n X_iX_i^T$ be the sample
covariance matrix. We further write $\hat\Sigma_{ij} := (\hat\Sigma
)_{ij}$ for the $(i,j)$-th element of $\hat\Sigma$, $(i,j)\in\mathcal
V\times\mathcal V$.
The graphical Lasso estimator $\hat\Theta$ \cite{glasso} is defined as
the solution to the optimization problem
%
\begin{equation}\label{alg}
\tag*{{(P1)}}
\hat\Theta:= \text{arg}\min_{\Theta\in S_{++}^p}\left\{ \text
{trace}(\Theta^T\hat\Sigma) - \log\text{det}(\Theta) + \lambda\|\Theta
\|_{1,\text{off}}\right\},
\end{equation}
where $\hat\Sigma$ is the sample covariance matrix $\hat\Sigma= \frac
{1}{n}\sum_{i=1}^n X_iX_i^T$ and $\|\cdot\|_{1,\text{off}}$ is the $\ell
_1$ off-diagonal penalty, $\|\Theta\|_{1,\text{off}}
= \sum_{i\not=j} |\Theta_{ij}|$.
When the data is normally distributed, \ref{alg} is equivalent to $\ell
_1-$penalized maximum likelihood for the precision matrix.

In our analysis, we rely on the results on rates of convergence derived
in \cite{ravikumar} for the graphical Lasso in supremum norm. The work
\cite{ravikumar} assumes an irrepresentability condition which is a
rather restrictive condition in the linear regression setting \cite
{vdgeer09}. However, other literature on the graphical Lasso \cite
{yuan} likewise assumes irrepresentability condition or otherwise
assumes $p/n\rightarrow0$ \cite{rothman}.
In the linear regression setting, irrepresentable conditions are
sufficient for variable selection \cite{vdgeer09,buhlmann,zhao}.

The analysis of rates of convergence of the graphical Lasso in \cite
{ravikumar} in addition considers certain functions of the true
precision matrix $\Theta^*$, which we now define.

Let $\kappa_{\Sigma^*}$ be the $\ell_\infty$ operator norm of the true
covariance matrix $\Sigma^*$, i.e.
\[
\kappa_{\Sigma^*} = \vertiii{\Sigma^*}_\infty= \max_i \sum_{j=1}^p
|\Sigma^*_{ij}|.
\]
The parameter $\kappa_{\Sigma^*}$ then measures the size of entries in
$\Sigma^*$.

\begin{example}\label{ex1}
Consider the T\"oplitz matrix $\Sigma^{*}_{ij} = \rho^{|i-j|}$ for
$i,j=1,\dots,p,$ where $|\rho|< 1.$ Then
$\kappa_{\Sigma^*}=({1-\rho^p})/{(1-\rho)}=\mathcal O(1)$ if $\rho$ is
bounded away from $1$.
\end{example}

We next consider the Hessian $\Gamma(\Theta)$ of the negative
log-likelihood function
$\ell(\Theta)=\text{tr}(\Theta^T\hat\Sigma)
-\log\text{det}(\Theta)$.
The entries of the gradient of $\ell$ are given by
(\cite{greene})
\[
\frac{\partial\ell(\Theta)}{\partial\Theta_{ij}} = \hat\Sigma
_{ij}-(\Theta^{-1})_{ij}.
\]
The Hessian matrix is then indexed by pairs of edges $((i,j),(k,l))$
and the $((i,j),(k,l))$-th entry takes the form
\begin{eqnarray*}
\frac{\partial^2 \ell(\Theta)}{\partial\Theta_{kl}\partial\Theta_{ij}}
=
\frac{\partial(\hat\Sigma_{ij}-(\Theta^{-1})_{ij} )}{\partial\Theta
_{kl}\partial\Theta_{ij }}
=e_i^T\Theta^{-1} e_k e_l^T \Theta^{-1}e_j =\Sigma_{ik}\Sigma_{lj},
\end{eqnarray*}
where $\Sigma=\Theta^{-1}.$
In matrix form, we obtain
\[
\Gamma(\Theta) = \Sigma\otimes\Sigma.
\]
By $(i,j)$-th column of $\Sigma\otimes\Sigma$ we refer to the
$p^2\times1$ vector $\Sigma\otimes\Sigma\text{vec}(e_ie_j^T)$ and $(i,j)$-th
row of $\Sigma\otimes\Sigma$ is its transpose. The $(i,j)$-th row of
$\Sigma\otimes\Sigma$ contains all mixed partial derivatives of $\ell$
with respect to $\Theta_{ij}$ and $\Theta_{kl}$ where $k,l=1,\dots,p$.
Note that $\Theta\otimes\Theta$ may be viewed as a four-dimensional tensor.

\smallskip
We impose some restrictions on the Hessian $\Gamma$ evaluated at the
true $\Theta^*$, $\Gamma^* := \Gamma(\Theta^*)$.
To this end, let us fix the following notation. For any two subsets $T$
and $T'$ of $\mathcal V\times\mathcal V$, we use $\Gamma^*_{TT'}$ to
denote the $|T|\times|T'|$ matrix with rows
and columns of $\Gamma^*$ indexed by $T$ and $T'$ respectively.

Consequently, define $\kappa_{\Gamma^*}$ to be the $\ell_\infty$
operator norm of the inverse of the matrix
\[
\Gamma_{SS}^* = [\Sigma^* \otimes\Sigma^*]_{SS} \in\mathbb R^{s\times s}.
\]
i.e., $\kappa_{\Gamma^*}=\vertiii{(\Gamma^*_{SS})^{-1}}_\infty.$

The parameter $\kappa_{\Gamma^*}$ then measures the size of entries in
$\Theta^*$ and assumptions on its growth are similar to sparsity
assumptions on $\Theta^*$.

\begin{example}\label{ex2}
The parameter $\kappa_{\Gamma^*}$ is difficult to track in general as
it involves inversion of a certain sub-matrix of the Hessian.
A tractable example is the situation when $\Theta^*$ is a block
diagonal matrix with blocks $B_1,\dots,B_k$ for some $1\leq k\leq p$
which only contain non-zero (although possibly arbitrarily small)
entries and the remaining off-diagonal entries of $\Theta^*$ are zero.
Suppose that the sizes of the blocks are $b_1,\dots,b_k$ and denote
$d:=\max_{i=1,\dots,k}b_i.$
This corresponds to a graph with $k$ completely connected but mutually
isolated subgraphs with maximum vertex degree $d.$
Using that block matrices can be easily inverted by inverting each
block separately (and using that $(A\otimes A)^{-1} = A^{-1}\otimes
A^{-1}$), some calculations give
\begin{eqnarray*}
\kappa_{\Gamma^*} = \max_{i=1,\dots,k} \vertiii{B_i\otimes B_i}_\infty
= \max_{i=1,\dots,k} \vertiii{B_i}^2_\infty.
\end{eqnarray*}
The size of $\kappa_{\Gamma^*}$ thus depends on the size of entries in
$B_i's.$ We clearly have the upper bound
\[
\kappa_{\Gamma^*} \leq d\max_{i=1,\dots,k} \max_{j=1,\dots
,b_i}(B_i^{jj})^2 \leq d \Lambda^2_{\max}(\Theta^*).
\]
Hence if the maximum eigenvalue of $\Theta^*$ is bounded, then $\kappa
_{\Gamma^*}=\mathcal O(d).$
This bound is attained for instance when all entries in some row of the
block of size $d\times d$ were bounded away from zero uniformly in $n$
and $\Lambda_{\max}(\Theta^*)=\mathcal O(1)$.

In the trivial case when $\Sigma^*$ and $\Theta^*$ are diagonal
matrices, we have $S =\{(i,i):i=1,\dots,p\}$ and
and $\kappa_{\Gamma^*}={\max_i (\Theta^*_{ii})^2}.$
Then clearly $\kappa_{\Gamma^*}\leq\Lambda_{\max}(\Theta^*).$

The most difficult situation is when $\Theta^*$ a single block, so $S=\{
(i,j):i,j=1,\dots,p\}$
and $\kappa_{\Gamma^*}= \vertiii{ (\Sigma^* \otimes\Sigma^*)^{-1}
}_\infty=
\vertiii{ \Theta^* \otimes\Theta^* }_\infty=
\vertiii{ \Theta^* }^2_\infty.$
If the entries in $\Theta^*$ are fast decaying, for instance $\Theta
_{ij}^* = \tau^{|i-j|}$, $|\tau|<1$
then
\[
\hspace*{90pt}\kappa_{\Gamma^*} = ({1-\tau^p})^2/({1-\tau})^2 =\mathcal O(1).\hspace*{90pt}\qed
\]
\end{example}

\begin{assump}[Irrepresentability condition]
\label{ir}
There exists $\alpha\in(0,1]$ such that
%
\begin{equation}
\max_{e \in S^c}\|\Gamma^*_{eS}(\Gamma^*_{SS})^{-1}\|_1 \leq1-\alpha.
\end{equation}
\end{assump}

Condition \ref{ir} is an analogy of the irrepresentable condition for
variable selection in linear regression \cite{vdgeer09}.
If we define the zero-mean edge random variables (\cite{ravikumar}) as
\[
Y_{(i,j)}:= X_i X_j - \mathbb E(X_iX_j),
\]
then the matrix $\Gamma^*$ corresponds to covariances of the edge
variables, in particular
$\Gamma^*_{(i,j),(k,l)} + \Gamma^*_{(j,i),(k,l)} = \text
{cov}(Y_{(i,j)},Y_{(k,l)})$.
The interpretation of \ref{ir} is that we require that no edge variable
$Y_{(j,k)}$ which is not included in the edge set $S$ is highly
correlated with variables in the edge set \cite{ravikumar}. The
parameter $\alpha$ then is a measure of this correlation with the
correlation growing when $\alpha\rightarrow0$.

Note that one may view $\Gamma^*$ as a four-dimensional tensor, and the
irrepresentability condition is then imposed on the sub-blocks of this
tensor.

\begin{remark}
The results obtained in \cite{ravikumar} (see also Lemma \ref{rates})
imply that under the irrepresentability condition \ref{ir},
the model $\hat S$ selected by the graphical Lasso satisfies $\hat
S\subseteq S$ with high probability. Moreover, under a beta-min
condition on the entries of the true precision matrix, Lemma \ref{rates} part (b)
implies exact variable selection, i.e. $\hat S = S$ with high probability.
Several works then suggest to use post-model selection methods, by
which we
refer to the two-step procedure resulting from first selecting
a model and then estimating the parameters in
the selected model (e.g. by maximum likelihood).
For estimation of regression coefficients in the linear model, simple
post-model selection methods have been proposed e.g. in \cite
{gausslasso,candes2007}. Other approaches using post-model
selection in a more involved way include e.g.
\cite{belloni1,buja}.
We mention that several concerns have been raised considering simple
post-model selection methods,
which are elaborated on in the papers \cite{leeb1,leeb2} or \cite{buja}.
The procedure we suggest in the present paper in principle does not
rely on model selection (see also Remark \ref{nodewise}).
The advantage of our procedure over post-model selection methods is
likely to arise in situations when there are small but non-zero
parameters and thus
the beta-min type condition which guarantees exact variable selection
is violated.
\end{remark}

\begin{assump}[Bounded eigenvalues]\label{eig}
There exists $L \asymp1$ such that
\[
{1}/{L} \leq\Lambda_{\min}(\Theta^*)\leq\Lambda_{\max}(\Theta^*) \leq{L}.
\]
\end{assump}

\begin{remark}
In our analysis to follow in Section \ref{sec:main}, we keep track of the quantities
$\kappa_{\Sigma^*}$ and $\kappa_{\Gamma^*}$ defined above and they
appear in the main
result (Theorem~\ref{res2}). Some examples where the behaviour of $\kappa_{\Sigma
^*}$ and $\kappa_{\Gamma^*}$ is tractable
were discussed in Examples \ref{ex1} and \ref{ex2}. An example of a situation when
$\kappa_{\Sigma^*}$ is bounded is for
instance the T\"oplitz covariance structure. It is not easy to see when
$\kappa_{\Gamma^*}$ is bounded, as
it involves inversion of a certain sub-matrix of the Hessian. This is
of similar difficulty as
verification of the irrepresentability condition (see Assumption \ref
{ir}), which is typically considered only on small
examples \cite{ravikumar} ($p = 4$), \cite{meinshausen} ($p = 4$).
\end{remark}

\section{Main results}\label{sec:main}
In this Section we present the main results which imply inference for
individual parameters of the precision matrix.
We suggest a way to modify the graphical Lasso estimator by removing
the bias term associated with the penalty.
To this end we consider the Karush-Kuhn-Tucker (KKT) conditions for the
graphical Lasso.
For any $\lambda_n>0$ and $\hat\Sigma$ with strictly positive diagonal
elements, the optimization problem \ref{alg} has a unique solution $\hat
\Theta_n\in S^p_{++}$ which is characterized by the KKT conditions
%
\begin{equation}\label{kktc}
\hat\Sigma-\hat\Theta^{-1}+\lambda\hat Z = 0,
\end{equation}
where the matrix $\hat Z$ belongs to the sub-differential of the
off-diagonal norm $\|\cdot\|_{1,\text{off}}$ evaluated at $\hat\Theta$
(Lemma 3 in \cite{ravikumar}).

First we ``invert'' the KKT conditions \eqref{kktc}
by multiplying them by the inverse of the Hessian of the negative
log-likelihood, i.e. $(\Gamma^*)^{-1} = (\Sigma^*\otimes\Sigma^*)^{-1} = {(\Sigma
^*)}^{-1}\otimes(\Sigma^{*})^{-1}=\Theta^* \otimes\Theta^*$ which may
be approximated by plugging in the graphical Lasso estimator to obtain
$\hat\Theta\otimes\hat\Theta$.
As noted in Section \ref{subsec:setup}, when $X_1,\dots,X_n$ are
Gaussian, there is a correspondence between $\Sigma^*\otimes\Sigma^*$
and the Fisher information matrix for $\Theta^*.$

By the properties of the Kronecker product \cite{greene}, this is
equivalent to multiplication of \eqref{kktc} by $\hat\Theta$ from left
and right.
\[
\hat\Theta\hat\Sigma\hat\Theta- \hat\Theta+\hat\Theta\lambda\hat
Z\hat\Theta=0.
\]
Denoting $W:=\hat\Sigma-\Sigma^*$ and rearranging yields
%
\begin{equation}\label{nvm}
\hat\Theta+ \hat\Theta\lambda\hat Z \hat\Theta- \Theta^* = -\Theta
^*W \Theta^* + \text{rem},
\end{equation}
where
%
\begin{equation}\label{remain}
\text{rem} :=-(\hat\Theta- \Theta^* )W\Theta^* - (\hat\Theta\hat\Sigma
-I)(\hat\Theta- \Theta^*).
\end{equation}
The term $\text{rem}$ is shown to be small under sufficient sparsity
(Lemma \ref{res1}) and the leading term $\Theta^*W \Theta^*$ is
(elementwise) asymptotically normal. This suggests to take the modified
non-sparse estimator $\hat\Theta+ \hat\Theta\lambda\hat Z \hat\Theta
$ as an estimator for $\Theta^*.$ Recall that by the KKT conditions
\eqref{kktc}, $\lambda\hat Z$ may be expressed as $\hat\Theta^{-1}-\hat
\Sigma.$
Hence define the de-sparsified graphical Lasso estimator as follows
%
\begin{equation}\label{dsgl}
\hat T := \hat\Theta+ \hat\Theta\lambda\hat Z\hat\Theta= 2\hat\Theta
- \hat\Theta\hat\Sigma\hat\Theta.
\end{equation}
The following auxiliary Lemma gives a bound for the remainder \eqref
{remain} under sub-Gaussian tail assumptions \ref{subg} and \ref{subgv}.
\begin{lema}\label{res1}
Suppose that $X_1,\dots,X_n\in\mathbb R^p$ are independent and
distributed as $X=(X^1,\dots,X^p)$ with $\mathbb EX = 0,$ $\text
{cov}(X) = \Sigma^*.$ Let $\Theta^*=(\Sigma^*)^{-1}$ exist and satisfy
the irrepresentability condition \ref{ir} with a constant $\alpha\in(0,1]$.
Let $\hat\Theta$ be the solution to the optimization problem {\ref
{alg}} with tuning parameter
$\lambda_n =\frac{8}{\alpha} \delta_n$ where
for some $\gamma>2,$
\[
\delta_n := 8(1+12K^2)\max_i\Sigma^*_{ii} \sqrt{2\frac{\log(4p^\gamma)}{n}},
\]
where $K$ is specified below.
Suppose that the sparsity assumption
\[
d \leq\frac{1}{6(1+8/\alpha) \max\{\kappa_{\Sigma^*}\kappa_{\Gamma
^*},\kappa_{\Sigma^*}^3\kappa_{\Gamma^*}^2 \} \delta_n}
\]
and assumption \ref{eig} are satisfied.
\begin{enumerate}[(i)]
\item
Suppose that $X_1,\dots,X_n\in\mathbb R^p$ satisfy \ref{subg} with
$K=\mathcal O(1)$.
Then it follows that
\begin{eqnarray*}
\|\emph{rem}\|_\infty
=\mathcal O_\mathbb P\left(\frac{1}{\alpha^2}\kappa_{\Gamma^*} \max
\{ {d}^{3/2} \log p/n,
\frac{1}{\alpha}\kappa_{\Gamma^*} d^2 (\log p/n)^{3/2}
\}\right)
.
\end{eqnarray*}
%
%
\item
Suppose that $X_1,\dots,X_n\in\mathbb R^p$ satisfy \ref{subgv} with
$K=\mathcal O(1)$. Then
\begin{eqnarray*}
\|\emph{rem}\|_\infty
=
\mathcal O_\mathbb P\left(\frac{1}{\alpha^2}\kappa^2_{\Gamma^*}
\kappa_{\Sigma^*} d\log p/n \right).
\end{eqnarray*}
\end{enumerate}
\end{lema}

The quantities $\kappa_{\Sigma^*},\kappa_{\Gamma^*}$ involved in Lemma
\ref{res1}
measure the size of entries in $\Sigma^*$ and $(\Gamma_{SS}^*)^{-1}$ as
discussed in Section \ref{subsec:setup}. The parameter $\alpha$
corresponds to the irrepresentability condition \ref{ir} and affects
the rates when it approaches zero.

If we assume the quantities in Lemma \ref{res1} are bounded, i.e.
$1/\alpha=\mathcal O(1)$, $\kappa_{\Sigma^*}=\mathcal O(1)$ and $\kappa
_{\Gamma^*}=\mathcal O(1)$
then the sparsity assumption reduces to
\[
d\leq\sqrt{n/\log p}
\]
and under \ref{subg} we have
\[
\|\text{rem}\|_\infty=\mathcal O_\mathbb P\left(d^{\frac{3}{2}} \frac
{\log p}{\sqrt{n}}\right),
\]
under \ref{subgv} we have
\[
\|\text{rem}\|_\infty=
\mathcal O_{\mathbb P}\left( d\frac{\log p}{\sqrt{n}} \right)
.
\]
Consequently, we establish asymptotic normality of each element $\hat T_{ij}$ of the de-sparsified estimator
in Theorem \ref{res2} below. Since our aim is inference about
individual elements of $\Theta^*,$ fix $(i,j)\in\mathcal V\times
\mathcal V$. The $k$-th column of $\Theta^*$ will be denoted by $\Theta
^*_{ k}\in\mathbb R^{p}$, $k=1,\dots,p$.
\begin{theorem}\label{res2}
Suppose that $X_1,\dots,X_n\in\mathbb R^p$ are independent and
distributed as $X=(X^1,\dots,X^p)$ with $\mathbb EX = 0,$ $\text
{cov}(X) = \Sigma^*$. Let $\Theta^*=(\Sigma^*)^{-1}$ exist, satisfy the
irrepresentability condition \ref{ir} with a constant $\alpha\in(0,1]$
and assumption \ref{eig}.
Let
\[
\sigma_{ij}^2 := \emph{Var}({\Theta^*_{ i}}^T X_{1}X_{1}^T \Theta^*_{ j})
\]
and suppose that $1/\sigma_{ij}=\mathcal O(1).$
Suppose that $\hat\Theta$ is the solution to the optimization problem
{\ref{alg}} with tuning parameter
$\lambda_n\asymp\sqrt{\log p/n}$.
Suppose the sparsity assumption under \ref{subg}
%
\begin{equation}\label{sp1}
d^{3/2}=o\left(
\frac{\sqrt{n}}{C_1\log p}
\right),
\end{equation}
where
\[
C_1 := \max\left\{
\frac{\kappa_{\Gamma^*}}{\alpha^2} ,
\frac{\kappa_{\Gamma^*}^2}{\alpha^{9/8}} n^{-1/4}(\log p)^{1/8},
\frac{
\max\{\kappa_{\Sigma^*}\kappa_{\Gamma^*}, \kappa_{\Sigma^*}^3\kappa
_{\Gamma^*}^2\}^{3/2}
}{\alpha^{3/2}}
(n \log p)^{-1/4}
\right\}.
\]
and under \ref{subgv}
%
\begin{equation}\label{sp2}
d = o \left(\frac{\sqrt{n}}{
C_2\log p
}
\right),
\end{equation}
where
\[
C_2:=\frac{1}{\alpha}\kappa_{\Gamma^*} \max\left\{
\frac{1}{\alpha}\kappa_{\Gamma^*}\kappa_{\Sigma^*},
\kappa_{\Sigma^*}(\log p)^{-1/2},
\kappa_{\Sigma^*}^3\kappa_{\Gamma^*}(\log p)^{-1/2}
\right\},
\]
is satisfied.
Let $\hat T$ be the de-sparsified graphical Lasso estimator defined in
\eqref{dsgl}.
Then under \ref{subg} with sparsity \eqref{sp1} or under \ref{subgv}
with sparsity \eqref{sp2} for all $(i,j) \in\mathcal V\times\mathcal
V$, it holds that
%
\begin{equation}\label{result}
\sqrt{n}({\hat T_{ij}-\Theta^*_{ij}})/\sigma_{ij} = Z_{ij}^n
+ o_{\mathbb P}(1),
\end{equation}
where $Z_{ij}^n$ converges weakly to $\mathcal N(0,1)$.
\end{theorem}
When the quantities $\kappa_{\Gamma^*}$, $\kappa_{\Sigma^*}$ and
$1/\alpha$ are assumed to be bounded,
then the sparsity assumptions of Theorem \ref{res2} reduce to
$d^{3/2}=o(\sqrt{n}/{\log p})$
under \ref{subg} and
$d=o(\sqrt{n}/{\log p})$
under \ref{subgv}.
The latter condition $d=o(\sqrt{n}/{\log p})$ is the same sparsity
assumption as required for construction of confidence intervals for
regression coefficients using the de-sparsified Lasso \cite{vdgeer13}.

The asymptotic variance $\sigma_{ij}$ in Theorem \ref{res2} is
typically unknown, so to construct confidence intervals one needs to
use a consistent estimator $\hat\sigma_{ij} >0$ for $\sigma_{ij}.$
For the case of Gaussian observations, we may easily calculate the
theoretical variance and plug in the estimate $\hat\Theta$ in place of
the unknown $\Theta^*$ as is displayed in Lemma \ref{var} below.
\begin{lema}\label{var} Suppose that assumption \ref{eig} is satisfied
and assume that
$X_1,\dots,$ $X_n\in\mathbb R^p$ are independent $\mathcal N(0,\Sigma
^*)$. Let $\hat\Theta$ be the graphical Lasso estimator, let $\lambda
\asymp\sqrt{\log p/n}$
and suppose the sparsity assumption
\[
d \leq\frac{\sqrt{n}}{\log p (1+8/\alpha) \max\{\kappa_{\Sigma^*}\kappa
_{\Gamma^*},\kappa_{\Sigma^*}^3\kappa_{\Gamma^*}^2 \} }.
\]
Then $\sigma_{ij}^2 = \Theta^*_{ii} \Theta^*_{jj} + {\Theta^*_{ij}}^2,$
$1/\sigma_{ij}=\mathcal O(1)$ and
for $\hat\sigma^2_{ij}:= \hat\Theta_{ii} \hat\Theta_{jj} + \hat\Theta_{ij}^2$
we have
\[
|\hat\sigma^2_{ij} - \sigma^2_{ij}| = \mathcal O_{\mathbb P}\left(
1/\alpha\kappa_{\Gamma^*}\sqrt{\log p/n}
\right)
.
\]
\end{lema}

Hence by Lemma \ref{var} under assumptions of Theorem \ref{res2} we
have $\sigma_{ij}/\hat\sigma_{ij}=o_{\mathbb P}(1)$ and we may replace
$\sigma_{ij}$ by $\hat\sigma_{ij}$.

Theorem \ref{res2} also implies convergence rates for the de-sparsified
graphical Lasso estimator $\hat T$ in supremum norm.
Under \ref{subgv} and assumptions of Theorem \ref{res2} we have the
upper bound
%
\begin{eqnarray}\nonumber
\|\hat T-\Theta^*\|_\infty&\leq&\|\Theta^*W\Theta^*\|_\infty+\|\text
{rem}\|_\infty\\
&=&\label{lev}
\mathcal O_{\mathbb P}\left(
\max\left\{
\sqrt{\frac{\log p}{n}},
\frac{1}{\alpha^2}\kappa^2_{\Gamma^*}\kappa_{\Sigma^*} \frac{\log p}{n}
\right\}
\right).
\end{eqnarray}
Assuming $\kappa_{\Gamma^*}$, $\kappa_{\Sigma^*}$ and $1/\alpha$
bounded implies
%
\begin{equation}\label{trates}
\|\hat T-\Theta^*\|_\infty=\mathcal O_{\mathbb P}(\max\{
\sqrt{\log p/n},
d\log p/n
\}
).
\end{equation}
Under sparsity $d=o(\sqrt{n}/\log p)$ we have $\|\hat T-\Theta^*\|
_\infty=\mathcal O_{\mathbb P}(\sqrt{\log p/n}).$

Consequently, under the conditions of Lemma \ref{var} thresholding
$\hat T_{ij}$ at level
$\Phi^{-1}(1-\frac{\alpha}{p(p-1)})\frac{\hat\sigma
_{ij}}{\sqrt{n}}$ for all $i,j$ will remove all zero entries
with probability $1-\alpha$ asymptotically.

\begin{remark}
The quantities $\kappa_{\Sigma^*}, \kappa_{\Gamma^*}$ and $\alpha$
involved in our analysis arise from the deterministic analysis of the
graphical Lasso as carried out in \cite{ravikumar}.
Provided that the quantities $\kappa_{\Sigma^*}, \kappa_{\Gamma^*}$ and
$1/\alpha$ remain bounded and assuming sub-Gaussianity \ref{subgv}, the
only additional restriction that arises from our analysis is the
sparsity restriction by a factor $\sqrt{n}$ which is needed to ensure
that the remainder term in Lemma \ref{res1} vanishes asymptotically. As
mentioned above, this is the same assumption which is needed to
establish asymptotic normality of the de-sparsified Lasso in linear
regression \cite{vdgeer13}.
Under these assumptions, the estimator $\hat T$ achieves optimal rate
of convergence under the assumed model which follows from \eqref
{trates} and the work
\cite{zhou}.
\end{remark}
\begin{remark}\label{nodewise}
One could consider the approach presented above for other initial
estimators of the precision matrix than the graphical Lasso.
An estimator of the precision matrix based on the nodewise regression
approach (\cite{meinshausen})
is outlined in \cite{skript}.
Asymptotic normality of the estimator in \cite{skript} may then be
obtained under
bounded eigenvalues of the true precision matrix,
row sparsity of $\Theta^*$ of small order $\sqrt{n}/\log p$
and assuming fourth-order moment conditions on the $X_i$'s (see~\cite{skript}).
To avoid digressions we do not elaborate on this alternative approach
in the present paper.
We note that the present analysis using the graphical Lasso requires in
addition to the conditions mentioned above the irrepresentability
condition. However, inspection of the proof of Theorem \ref{res2} reveals that
it is rather the $\ell_1-$norm oracle rates that are needed, but only
results assuming the irrepresentability condition are available in the
literature on the graphical Lasso at the moment. It is as of yet not
clear whether the irrepresentability condition is also necessary for
obtaining oracle rates for the graphical Lasso.
We also refer here to \cite{vdgeer14} where the results of \cite
{ravikumar} are extended using an irrepresentable condition on the
small (not necessarily zero) entries of $\Theta^*$.
\end{remark}
%

\section{Empirical Results}\label{sec:emp}
\subsection{Simulation Study}

In this part we illustrate the theoretical results on simulated data
and demonstrate the performance of the proposed estimator on inference,
giving a comparison to some alternative methodologies.
To this end, we consider sparse Gaussian graphical models which may be
fully specified by a precision matrix $\Theta^*.$
Thus the random sample is distributed as $X\sim\mathcal N(0,\Sigma
^*),$ where $\Theta^* = (\Sigma^*)^{-1}$.

We consider the chain graph on $p$ vertices where the maximum vertex
degree is by definition restricted to $d=2$.
The corresponding precision matrix
$\Theta^*$ is a tridiagonal matrix, $\Theta^* = \text{tridiag}(\rho
,1,\rho)$ for a given $\rho>0.$
The cardinality of the active set is then $3p-2.$
To solve the graphical Lasso program \ref{alg}, we have used the
implemented procedure \texttt{glasso} of Friedman et al. \cite{glasso}.

Denote by $\hat\Theta$ the graphical Lasso estimator and by $\hat T$
the de-sparsified graphical Lasso estimator.
In figure \ref{fig:hist}
we report histograms of $\sqrt{n}(\hat T-\Theta^*_{ij})/\hat\sigma
_{ij}$ for $(i,j)\in\{(1,1),(1,2),$ $(1,3),(1,4)\}$
with the density of $\mathcal N(0,1)$ superimposed.
%
\begin{figure}[h!]
\centering
\textbf{Asymptotic normality}\\[6pt]
\includegraphics[width=\textwidth]{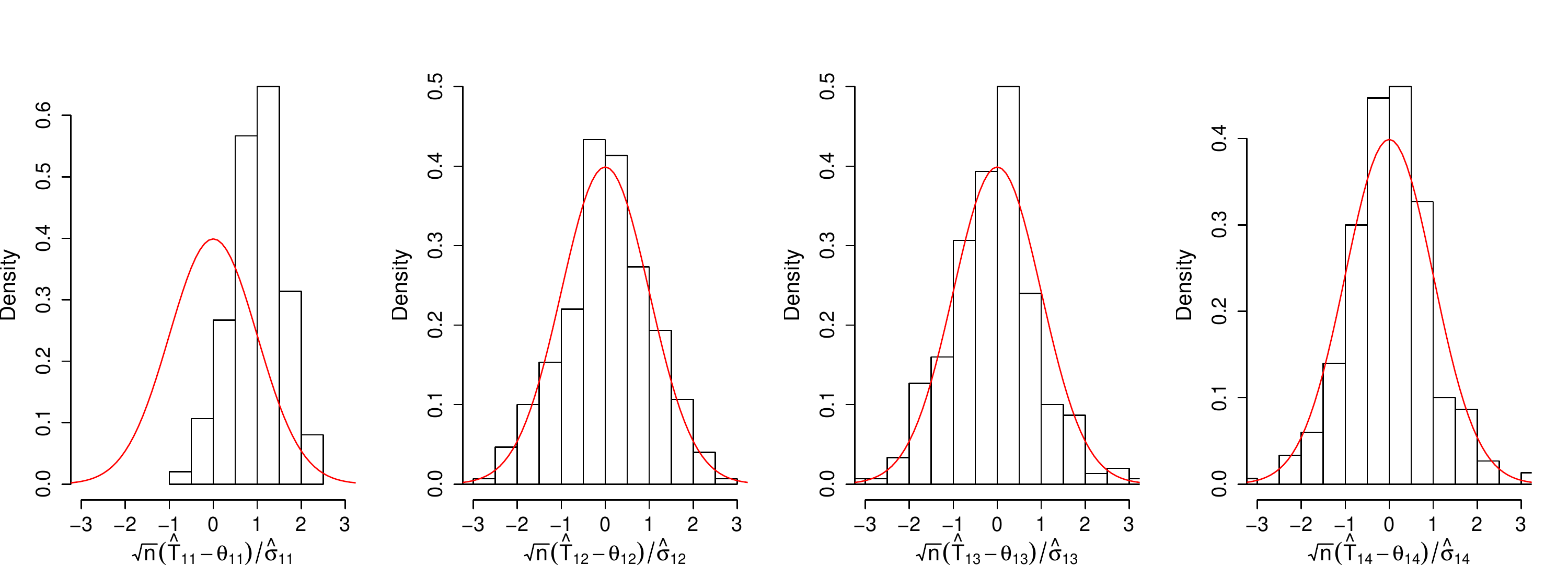}
\caption{Histograms for $\sqrt{n}(\hat T_{ij}-\Theta^*_{ij})/\hat\sigma_{ij}$,
$(i,j)\in\{(1,1),(1,2),(1,3),(1,4)\}$. The sample size was $n=500$ and
the number of parameters $p=100$. The de-sparsified graphical Lasso
estimator was calculated $300$ times.
The model was the chain graph with $\rho=0.3$.}\label{fig:hist}
\end{figure}
By Theorem \ref{res2}, it follows that\vadjust{\goodbreak} the $(1-\alpha)100\%$ asymptotic
confidence interval for $\Theta^*_{ij}$ is given by
\[
I_{ij} \equiv I_{ij}(\hat\Theta_{ij},\alpha,n) := [\hat T_{ij} - \Phi
^{-1}(1-\alpha/2)\frac{\sigma_{ij}}{\sqrt{n}}, \hat T_{ij} + \Phi
^{-1}(1-\alpha/2)\frac{\sigma_{ij}}{\sqrt{n}}],
\]
where we replace the unknown variance $\sigma_{ij}^2$ by the plug-in
estimate $\hat\sigma_{ij}^2 = \hat\Theta_{ii}\hat\Theta_{jj}+\hat\Theta
_{ij}^2$ (Lemma \ref{var}).
For each parameter $\Theta^*_{ij}$, the probability that the true value
$\Theta^*_{ij}$ is covered by the confidence interval
was estimated by its empirical version,
$\hat\alpha_{ij} := \mathbb P_N \mathbf1_{\{\Theta_{ij}^* \in
I_{ij,\alpha}\}}.$
The number of iterations used to calculate the estimates $\hat\alpha
_{ij}$ was set to $N=50.$
Next for a set $A\subset\mathcal V\times\mathcal V$ define the
average coverage over the set $A$ as
\[
\text{Avgcov}_A := \frac{1}{|A|}\sum_{(i,j)\in A} \hat\alpha_{ij}.
\]
After obtaining the estimates $\hat\alpha_{ij}$, we have averaged them
over the sets $S$ and $S^c$ to obtain $Avgcov_{S}$ and $Avgcov_{S^c},$
respectively.
Similarly, we have calculated the average length of the confidence
interval for each parameter $\Theta^*_{ij}$ from $N=50$ iterations and
again averaged these over the sets $S$ and $S^c$ to obtain
$Avglength_{S}$ and $Avglength_{S^c}$.

We compare the performance of the graphical Lasso-based confidence
intervals with three other methods.
The first method is based on the oracle maximum likelihood estimator
with the non-zero set $S$ pre-specified. This method only serves as a
theoretical benchmark as it is asymptotically efficient if the true
non-zero set $S$ is known.
The second method is a post-model selection method: the maximum
likelihood estimator is applied to the model selected by the graphical
Lasso (see Lemma \ref{rates}: under the irrepresentability condition, the
graphical Lasso selects a model $\hat S\subseteq S$).
Then the confidence intervals are constructed using asymptotic
normality of the maximum likelihood estimator.
The third method is based on the sample covariance matrix $\hat\Sigma$
which is the MLE estimator for $\Sigma^*$.
Its inverse is then the maximum likelihood estimator for the precision
matrix $\Theta^*$.
In the fixed $p$ setting, the inverse sample covariance matrix $\hat
\Sigma^{-1}$ is thus an asymptotically normal and efficient estimator
of the precision matrix when the observations are Gaussian. This allows
for construction of confidence intervals in the classical way, using
asymptotic normality of $\hat\Theta=\hat\Sigma^{-1}$, hence the
confidence interval for $\Theta^*_{ij}$ is given by
$\hat\Theta_{ij}\pm\Phi^{-1}(1-\alpha/2)\frac{\hat\sigma_{ij}}{\sqrt
{n}}$.

\begin{table}[t!!]
\caption{The tables above show a comparison of four methods for
construction of confidence intervals: the de-sparsified graphical Lasso
(``De-sp. graphical Lasso''), the maximum likelihood estimator with
specified set $S$ (``MLE with specified $S$''), the maximum likelihood
estimator based on the non-zero set $\hat S$ selected by the graphical
Lasso (``MLE based on $\hat S$'') and an estimator based on the sample
covariance matrix $\hat\Sigma$ (``Sample covariance'').
The table shows average coverages and average lengths of the
constructed confidence intervals over the sets $S$ and $S^c$ (where
applicable). For the method ``MLE based on $\hat S$'', the reported
averages are over $\hat S\cap S.$
The regularization parameter for the graphical Lasso was chosen
$\lambda= \sqrt{\frac{\log p}{n}}$ in all simulations.
The true precision matrix corresponds to a chain graph with $p$
vertices and $\rho$ equal to $0.3,0.4,0.2$ for settings S1,S2,S3,
respectively}\label{fig:tab1}
\smallskip\textbf{Estimated coverage probabilities and lengths}\par\smallskip
\begin{tabular}{ccccc}
\hline
\multirow{2}{*}{ \textbf{S1.} $p=80,n=250,\rho=0.3$ } & $S$ & $S$ & $S^c$ & $S^c$ \\
& Avgcov & Avglength & Avgcov & Avglength \\
\hline\\[-2ex]
De-sp. graphical Lasso & 0.934 & 0.247 & 0.972 & 0.215 \\
MLE with specified $S$ & 0.940 & 0.308 & -- &-- \\
MLE based on $\hat S$ & 0.887 & 0.325 & -- & -- \\
Sample covariance & 0.459 & 0.428 & 0.897 & 0.367 \\
[0.1cm] \hline     \vspace*{3pt}
\end{tabular}
\begin{tabular}{ccccc}\hline
\multirow{2}{*}{\textbf{S2.} $p=100,n=200,\rho=0.4$ } & $S$ & $S$ & $S^c$ & $S^c$ \\
& Avgcov & Avglength & Avgcov & Avglength \\
\hline\\[-2ex]
De-sp. graphical Lasso & 0.925 & 0.288 & 0.974 & 0.250 \\
MLE with specified $S$ & 0.945 & 0.349 & -- &-- \\
MLE based on $\hat S$ & 0.856 & 0.374 & -- & -- \\
Sample covariance & -- & -- & -- & -- \\
[0.1cm] \hline       \vspace*{3pt}
\end{tabular}
\begin{tabular}{ccccc}\hline
\multirow{2}{*}{ \textbf{S3.} $p=100,n=200,\rho=0.2$ } & $S$ & $S$ &
$S^c$ & $S^c$ \\
& Avgcov & Avglength & Avgcov & Avglength \\
\hline\\[-2ex]
De-sp. graphical Lasso & 0.951 & 0.301 & 0.964 & 0.263 \\
MLE with specified $S$ & 0.943 & 0.357 & -- &-- \\
MLE based on $\hat S$ & 0.747 & 0.328 & -- & -- \\
Sample covariance & -- & -- & -- & -- \\
[0.1cm] \hline
\end{tabular}\vspace*{3pt}
\end{table}

The results are reported in Tables \ref{fig:tab1} and \ref{fig:tab2}. The de-sparsified graphical
Lasso performs well also when compared to the oracle (``MLE with
specified $S$''). In Table \ref{fig:tab1}, settings S2 and S3 differ only in the
value of $\rho.$
The de-sparsified graphical Lasso performs well for both of these
settings, while the the post-model selection method (``MLE based on
$\hat S$'') shows lower coverage for setting S3, where $\rho$ is
comparable in magnitude to the noise level.

In table \ref{fig:tab2},
the sample size is kept fixed at $n=100$ while the dimension of the
parameter is increased, hence
we observe lower coverage on $S$ for large values of $p$, as
expected.\vadjust{\goodbreak}

\smallskip\noindent
\textit{The choice of the regularization parameter}\quad
Theory implies that the correct choice of the regularization parameter
satisfies $\lambda_n \asymp\sqrt{\frac{\log p}{n}}$. Lemma \ref{res1}
gives an explicit prescription for $\lambda$, however, this
theoretically obtained value of $\lambda$ is very large.
For all the numerical experiments, we have chosen $\lambda_n=\sqrt{\frac
{\log p}{n}}$.

\subsection{Real data experiment}
We consider a dataset about riboflavin (vitamin $B_2$) production by
bacillus subtilis without the response variable. The dataset is
available from the \texttt{R} package \texttt{hdi}.

The dataset contains observations of $p=4088$ logarithms of gene
expression levels from $n=71$ genetically engineered mutants of
bacillus subtilis. We are interested in modeling the conditional
independence structure of the covariates (logarithms of gene expression
levels) and we try to estimate the associated graphical model using the
de-sparsified graphical Lasso. We only consider the first $500$
covariates which have the highest variances.

In the first step, we split the sample and use 10 randomly chosen
observations to estimate the variances of the $500$ variables. With the
estimated variances, we scale the design matrix containing the
remaining $61$ observations.

We calculate the graphical Lasso using the tuning parameter as in
the\break
simulations, $\lambda=\sqrt{\log p/n}$,
and hence calculate the de-sparsified graphical Lasso.

We threshold the de-sparsified graphical Lasso at level
$\Phi^{-1}(1-\frac{\alpha}{p(p-1)}) \hat\sigma_{ij}/\sqrt
{n},$ where $\alpha=0.05$ and $\hat\sigma_{ij}^2= \hat\Theta_{ii}\hat
\Theta_{ii}+\hat\Theta_{ij}^2$ is an estimate of the asymptotic
variance calculated under the assumption of normality and using the
graphical Lasso estimator $\hat\Theta.$
We identify $5$ edges as significant.

For independently permuted variables (for each variable, a different
per\-mutation is used), the conditional dependencies are broken (the
truth is\break  the empty graph), and the de-sparsified graphical Lasso
correctly detects
zero  edges.\looseness=1

\begin{table}[t]
\caption{A table showing the performance of the
de-sparsified graphical Lasso
for number of parameters $p$ taking values $100,200,300,500$ and
$n=100$. The regularization parameter was chosen $\lambda= \sqrt{\frac
{\log p}{n}}$ in all simulations. The constant $\rho$ is $0.3$ in the
definition of $\Theta^*$}\label{fig:tab2}
\medskip\textbf{Estimated coverage probabilities and lengths}\par\smallskip
\begin{tabular}{ccccc}\hline
{\multirow{2}{*}{Chain graph}} & $S$ & $S$ & $S^c$ & $S^c$ \\
& Avgcov & Avglength & Avgcov & Avglength \\
\hline
{p = 100} & 0.931 & 0.401 & 0.978 & 0.348 \\
{p = 200} & 0.917 &0.400 & 0.984 & 0.349 \\
{p = 300} & 0.893 &0.401 & 0.988 & 0.349 \\
{p = 500} & 0.832 & 0.401 & 0.988 & 0.350\\
\hline
\end{tabular}     \vspace*{3pt}
\end{table}
%

\section{Proofs}\label{sec:proofs}

\begin{lema} \label{res0}
Suppose that $X_1,\dots,X_n\in\mathbb R^p$ are independent and
distributed as $X=(X^1,\dots,X^p)$ with $\mathbb EX = 0,$ $\text
{cov}(X) = \Sigma^*.$ Let $\Theta^*=(\Sigma^*)^{-1}$ exist and satisfy
the irrepresentability condition \ref{ir} with a constant $\alpha\in(0,1]$.
Let $\hat\Theta_n$ be the solution to the optimization problem {\ref
{alg}} with tuning parameter
$\lambda_n =\frac{8}{\alpha} \delta_n$ for some $\delta_n>0$ and
suppose that the sparsity assumption
%
\begin{equation}\label{sparsity.ravikumar}
d \leq\frac{1}{6(1+8/\alpha) \max\{\kappa_{\Sigma^*}\kappa_{\Gamma
^*},\kappa_{\Sigma^*}^3\kappa_{\Gamma^*}^2 \} \delta_n}
\end{equation}
is satisfied.
Then
on the set
$\mathcal T_n = \{\|\hat\Sigma-\Sigma^*\|_\infty<\delta_n\},$ we have

\smallskip\noindent\underline{Bound I}
%
\begin{eqnarray}\label{rem.sp1}
\|\emph{rem}\|_\infty
= \mathcal O\biggl(
\frac{1}{\alpha}\kappa_{\Gamma^*}\max\{
{d} \delta_n \|\Theta^* W\|_\infty,
\frac{1}{\alpha}\kappa_{\Gamma^*} d^2 \delta^3_n,
\frac{1}{\alpha}\kappa_{\Gamma^*}\kappa_{\Sigma^*} d\delta^2_n
\}
\biggr)
\end{eqnarray}
%
\smallskip\underline{Bound II}
%
\begin{eqnarray}\label{rem.sp2}
\|\emph{rem}\|_\infty
=\mathcal O\biggl(
\frac{1}{\alpha} \kappa_{\Gamma^*}
\max\{{d} \delta_n \|\Theta^* W\|_\infty,
\frac{1}{\alpha^2} \kappa_{\Gamma^*} d^2 \delta^3_n,
\frac{1}{\alpha} \Lambda_{\max}(\Theta^*) {d}^{3/2} \delta^2_n
\}\biggr).
\end{eqnarray}
\end{lema}

\begin{proof}[Proof of Lemma \ref{res0}]
The KKT conditions for the optimization problem \ref{alg} read
%
\begin{equation}\label{kkt}
\hat\Sigma- \hat\Theta^{-1} + \lambda\hat Z =0,
\end{equation}
where the matrix $\hat Z$ is the sub-differential of $\|.\|_{1,\text
{off}}$ at the optimum $\hat\Theta$.
Multiplying \eqref{kkt} by $\hat\Theta$ from both sides (which is
equivalent to multiplying the vectorized equation \eqref{kkt}
by $\hat\Theta\otimes\hat\Theta$), we obtain
\[
\hat\Theta\hat\Sigma\hat\Theta- \hat\Theta+ \hat\Theta\lambda\hat Z
\hat\Theta=0.
\]
Adding $\hat\Theta-\Theta^*$ to both sides and rearranging gives
%
\begin{equation}\label{nvm2}
\underbrace{\hat\Theta+ \hat\Theta\lambda\hat Z \hat\Theta}_{\hat T}
- \Theta^* = -\Theta^*(\hat\Sigma-\Sigma^*) \Theta^* + \text{rem},
\end{equation}
where, denoting $W:=\hat\Sigma- \Sigma^*$, we have
%
\begin{equation}\label{sal}
\text{rem} :=-(\hat\Theta- \Theta^* )W\Theta^* - (\hat\Theta\hat\Sigma
-I)(\hat\Theta- \Theta^*).
\end{equation}

\smallskip\noindent\underline{Bound I}
\begin{eqnarray*}
\|\text{rem}\|_\infty&\leq& \| (\hat\Theta- \Theta^* )W\Theta^* \|
_\infty+
\|(\hat\Theta\hat\Sigma-I)(\hat\Theta- \Theta^*) \|_\infty
\\
&\leq&
\vertiii{ \hat\Theta- \Theta^* }_\infty\|W\Theta^* \|_\infty
+
\|\hat\Theta\hat\Sigma-I\|_\infty\vertiii{\hat\Theta- \Theta^*
}_\infty
\end{eqnarray*}
We can bound
\begin{eqnarray*}
\|\hat\Sigma\hat\Theta- I\|_\infty
&=& \|(\hat\Sigma-\Sigma^*)(\hat\Theta-\Theta^*) + \Sigma^*(\hat\Theta
-\Theta^*) + (\hat\Sigma-\Sigma^*)\Theta^*\|_\infty
\\
&\leq&
\|\hat\Sigma-\Sigma^*\|_\infty\vertiii{\hat\Theta-\Theta^*}_\infty
+\vertiii{ \Sigma^*}_\infty\|\hat\Theta-\Theta^*\|_\infty + \|W\Theta
^* \|_\infty
\end{eqnarray*}
Hence
\begin{eqnarray*}
\|\text{rem}\|_\infty&\leq&
\underbrace{
\vertiii{ \hat\Theta- \Theta^* }_\infty\|W\Theta^* \|_\infty}_{\text{rem}_1}
+
\underbrace{
\|\hat\Sigma-\Sigma^*\|_\infty\vertiii{\hat\Theta-\Theta^*}^2_\infty
}_{\text{rem}_2}
\\
&&+\underbrace{
\vertiii{ \Sigma^*}_\infty\|\hat\Theta-\Theta^*\|_\infty\vertiii{\hat
\Theta- \Theta^* }_\infty
}_{\text{rem}_3}
+
\underbrace{
\|W\Theta^* \|_\infty\vertiii{\hat\Theta- \Theta^* }_\infty}_{\text{rem}_1}
\end{eqnarray*}
In what follows, condition on the event $\mathcal T_n=\{\|\hat\Sigma
-\Sigma^*\|_\infty\leq\delta_n\}$.
Note that $\vertiii{\Theta^*}_\infty= \max_i \|\Theta_{i}\|_1 \leq
\max_i \|\Theta_{ i}\|_2\sqrt{d} \leq\Lambda_{\max}(\Theta^*)\sqrt{d}$.

By Lemma \ref{rates}, part (\ref{zeros}), on $\mathcal T$ it holds that
$\hat\Theta_{S^c} = \Theta^*_{S^c}$.
Thus $\hat\Theta$ has at most $d$ nonzero entries per row.
Hence it follows
%
\begin{equation}\label{boun2}
\vertiii{\Delta}_\infty= \vertiii{\hat\Theta- \Theta^*}_\infty\leq
d\|\hat\Theta-\Theta^*\|_\infty.
\end{equation}
Next by Lemma \ref{rates}, part (\ref{rate}), we have the bound
\[
\|\hat\Theta- \Theta^*\|_\infty\leq2(1+8/\alpha)\kappa_{\Gamma^*}
\delta_n.
\]
We obtain
\begin{eqnarray*}
\|\text{rem}_1\|_\infty
&\leq&
\|\Theta^* W\|_\infty\vertiii{\Delta}_1
\leq
2(1+8/\alpha)\kappa_{\Gamma^*} {d} \delta_n \|\Theta^* W\|_\infty\\
\|\text{rem}_2\|_\infty&\leq&
\vertiii{\Delta}_1^2 \|W\|_\infty
\leq
4(1+8/\alpha)^2\kappa_{\Gamma^*}^2 d^2\delta^3_n\\
\|\text{rem}_3\|_\infty& \leq&
\|\Delta\|_\infty\vertiii{\Delta}_1 \vertiii{\Sigma^*}_1
\leq
4(1+8/\alpha)^2\kappa_{\Gamma^*}^2\kappa_{\Sigma^*}d \delta^2_n .
\end{eqnarray*}
Hence conditioned on $\mathcal T,$
\begin{eqnarray*}
\|\text{rem}\|_\infty& \leq&
4\max\{
2(1+8/\alpha)\kappa_{\Gamma^*} {d} \delta_n \|\Theta^* W\|_\infty, \\
&& 
4(1+8/\alpha)^2\kappa_{\Gamma^*}^2 d^2\delta^3_n, \\
&& 
4(1+8/\alpha)^2\kappa_{\Gamma^*}^2\kappa_{\Sigma^*} d\delta^2_n\}\\
&=& \mathcal O\left(
\frac{1}{\alpha}\kappa_{\Gamma^*}\max\{
{d} \delta_n \|\Theta^* W\|_\infty,
\frac{1}{\alpha}\kappa_{\Gamma^*} d^2 \delta^3_n,
\frac{1}{\alpha}\kappa_{\Gamma^*}\kappa_{\Sigma^*} d\delta^2_n
\}\right).
\end{eqnarray*}

\smallskip\noindent\underline{Bound II}\\
Observe that
\[
\vertiii{\hat\Theta}_1 \leq \vertiii{\hat\Theta-\Theta^*}_1 +\vertiii
{\Theta^*}_1
\leq 2(1+8/\alpha)\kappa_{\Gamma^*} d \delta_n +\sqrt{d} \Lambda_{\max
}(\Theta^*).
\]
By \eqref{sal} and using the KKT conditions we have
\begin{align*}
\hspace*{1pt}\|\text{rem}\|_\infty={}& \|\Delta W\Theta^* - \hat\Theta(\hat\Sigma
-\hat\Theta^{-1}) \|_\infty
\\
 \leq{}&
\|\Delta W\Theta^* \|_\infty+ \vertiii{\hat\Theta}_1 \|\lambda\hat Z\|
_\infty\|\Delta\|_\infty
\\
\leq{}&
2(1+8/\alpha)\kappa_{\Gamma^*} {d} \delta_n \|\Theta^* W\|_\infty
\\
&{}+
16 \frac{8}{\alpha}4(1+8/\alpha)^2\kappa_{\Gamma^*}^2 d^2 \delta^3_n +
2 \frac{8}{\alpha} (1+8/\alpha)\Lambda_{\max}(\Theta^*)\kappa_{\Gamma
^*} {d}^{3/2} \delta^2_n
\\
={}&\mathcal O\left(
\frac{1}{\alpha} \kappa_{\Gamma^*}
\max\{{d} \delta_n \|\Theta^* W\|_\infty,
\frac{1}{\alpha^2} \kappa_{\Gamma^*} d^2 \delta^3_n,
\frac{1}{\alpha} \Lambda_{\max}(\Theta^*) {d}^{3/2} \delta^2_n
\}
\right)
.\qedhere\hspace*{1pt}
\end{align*}
\end{proof}

\begin{proof}[Proof of Lemma \ref{res1}]

\noindent
(i) Under \ref{subg}, using Lemma \ref{hp} and by assumption \ref{eig}
we have
\begin{eqnarray*}
\|\Theta^* W \|_\infty&\leq& \vertiii{\Theta^*}_{\infty} \|W\|_\infty
=\max_{i=1,\dots,p} \|\Theta^*_i\|_1 \|W\|_\infty
\\
&\leq& \sqrt{d}\|\Theta_i^*\|_2 \|W\|_\infty=\mathcal O_{\mathbb
P}(\sqrt{d\log p/n}).
\end{eqnarray*}
Then by Lemma \ref{res0}, bound II,
\begin{eqnarray*}
\|\text{rem}\|_\infty
&=&\mathcal O_{\mathbb P}\left(
\frac{1}{\alpha} \kappa_{\Gamma^*}
\max\{{d} \delta_n \|\Theta^* W\|_\infty,
\frac{1}{\alpha^2} \kappa_{\Gamma^*} d^2 \delta^3_n,
\frac{1}{\alpha} \Lambda_{\max}(\Theta^*) {d}^{3/2} \delta^2_n
\}
\right)\\
& =&
\mathcal O_{\mathbb P}\left(\frac{1}{\alpha^2}\kappa_{\Gamma^*} \max\{
{d}^{3/2} \log p/n,
\frac{1}{\alpha}\kappa_{\Gamma^*} d^2 (\log p/n)^{3/2}
\}\right)
.
\end{eqnarray*}
(ii) Under \ref{subgv} and \ref{eig}, by Lemma \ref{conc} we have
\[
\|\Theta^* W \|_\infty=\max_{i,j=1,\dots,p} |\Theta^*_i We_j|
=\mathcal O_{\mathbb P}(\sqrt{\log p/n}).
\]
Hence using Lemma \ref{res0}, bound I and Lemma \ref{hp} we have
\begin{align*}
\|\text{rem}\|_\infty&= \mathcal O_{\mathbb P}\left(
\frac{1}{\alpha}\kappa_{\Gamma^*}\max\{
{d} \log p/n ,
\frac{1}{\alpha}\kappa_{\Gamma^*} d^2 (\log p/n)^{3/2},
\frac{\kappa_{\Gamma^*}\kappa_{\Sigma^*}}{\alpha} d\log p/n
\} \right)\\
&\stackrel{(a)}{=}
\mathcal O_{\mathbb P}\left(\frac{1}{\alpha^2}\kappa^2_{\Gamma^*}
\kappa_{\Sigma^*} d\log p/n \right).
\end{align*}
Step (a) above follows by the sparsity assumption \eqref
{sparsity.ravikumar} and since $\kappa_{\Gamma^*} \gtrsim1$. The
latter statement follows by \ref{eig}, using that $\Lambda_{\min
}((\Gamma^*_{SS})^{-1}) = 1/\Lambda_{\max}(\Gamma^*_{SS})\geq
1/\Lambda_{\max}(\Gamma^*)$ and that the eigenvalues of Kronecker
product satisfy $\Lambda_{\max}(\Sigma\otimes\Sigma)=\Lambda^2_{\max
}(\Sigma)$.
\end{proof}

\begin{proof}[Proof of Theorem \ref{res2}]
By Lemma \ref{res1}, under the sub-Gaussianity condition \ref{subg} or
\ref{subgv}, under the appropriate sparsity assumptions (combining
\eqref{sparsity.ravikumar} and \eqref{rem.sp1}, \eqref{rem.sp2})
it holds that $\|\text{rem}\|_\infty= o_\mathbb P(1).$
Hence for every $(i,j)$ it holds that
%
\begin{eqnarray}\label{elem}
\sqrt{n}(\hat T_{ij}-\Theta^*_{ij})
=\frac{1}{\sqrt{n}}\sum_{k=1}^n ({\Theta^*_{ i}}^T X_{k} {\Theta^*_{
j}} ^T X_{k} -\Theta^*_{ij})+ o_{\mathbb P}(1).
\end{eqnarray}
It remains to prove that the scaled summation term in (\ref{elem})
weakly converges to the normal distribution.
To this end, define
\[
Z_{ij,k}:= {\Theta^*_{i}}^T X_{k} {\Theta^*_{ j}} ^T X_{k}-\Theta^*_{ij}.
\]
For each $n$ fixed, $Z_{ij,1},\dots,Z_{ij,n}$ are identically
distributed r.v.s with mean $\mathbb E(Z_{ij,k})= (\Theta^*_{ i})^T
\Sigma^* \Theta^*_{ j} -\Theta_{ij}= e_i^T \Theta^*\Sigma^*\Theta^* e_j
- \Theta^*_{ij} = 0$.

Since each $X_{k}$ has sub-Gaussian elements, the variance
$\sigma_{ij}^2=\text{Var}(Z_{ij,k})= \text{Var} ({\Theta^*_{ i}}^T
X_{k} {\Theta^*_{ j}} ^T X_{k})$ is finite.
Denote $S_n = \sum_{k=1}^n Z_{ij,k}$. Then $s_n:=\text{var}(S_n) =
n\sigma^2_{ij}.$
Dividing (\ref{elem}) by $\sigma_{ij}>0,$ we obtain
\[
\sqrt{n}{(\hat T_{ij}-\Theta^*_{ij})}/{\sigma_{ij}} =
S_n/s_n
+
{ o_{\mathbb P}(1)}/{\sigma_{ij}}.
\]
First note that by the assumption $1/\sigma_{ij}=\mathcal O(1)$ we have
$ o_{\mathbb P}(1)/\sigma_{ij} = o_{\mathbb P}(1).$

\smallskip\noindent(i) Sub-Gaussian design \ref{subg}

To show $S_n/s_n\rightsquigarrow\mathcal N(0,1)$, we check the
Lindeberg condition, i.e. for all $\varepsilon>0$
\[
\lim_{n\rightarrow\infty} \frac{1}{s_n^2}
\sum_{k=1}^n \mathbb E(Z_{ij,k}^2\mathbf1(|Z_{ij,k}|>\varepsilon s_n))=0.
\]

\smallskip
Observe that since $\{Z_{ij,k}\}_{k=1}^n$ are identically distributed
for $n$ fixed and $s_n^2 = n\sigma_{ij}^2$,
\begin{eqnarray*}
\sum_{k=1}^n \frac{1}{s_n^2}\mathbb E(Z_{ij,k}^2\mathbf
1(|Z_{ij,k}|>\varepsilon s_n))
=
\frac{1}{\sigma_{ij}^2} \mathbb E(Z_{ij,1}^2 \mathbf
1(|Z_{ij,1}|>\varepsilon s_n)).
\end{eqnarray*}
Consequently, it remains to show that $\lim_{n\rightarrow\infty}
\mathbb E(Z_{ij,1}^2 \mathbf1(|Z_{ij,1}|>\varepsilon s_n)) = 0.$
For any $c>0$ we may rewrite
\[
c = \int_{0}^\infty\mathbf1_{c >t}dt.
\]
Therefore, applying the last observation and by Fubini's theorem we obtain
\begin{eqnarray*}
\int_{\Omega} |X|^2\mathbf1_{|X|>a} d\mathbb P
= a^2 \mathbb P( |X|>a) + 2\int_{a}^\infty u \mathbb P( |X|> u ) du.
\end{eqnarray*}
Then it follows
\begin{eqnarray*}
\label{perpartes}
\mathbb E(Z_{ij,1}^2 \mathbf1(|Z_{ij,1}|>\varepsilon\sigma_{ij} \sqrt{n}))
& \leq&
\varepsilon^2 \sigma_{ij}^2 {n} \mathbb P(|Z_{ij,1}|>\varepsilon\sigma
_{ij} \sqrt{n}) \\
&& + \int_{\varepsilon\sigma_{ij} \sqrt{n}}^\infty x \mathbb P(|Z_{ij,1}|>x)dx.
\end{eqnarray*}
To show that the limit of the right-hand side of the last inequality is
$0$ for $n\rightarrow\infty$, we use Lemma
\ref{subexp}, by which it follows that $Z_{ij,1}$ satisfies a tail
bound $\mathbb P(|Z_{ij,1}|>t)\leq4 de^{-\frac{t}{c_1 d}}$.
For a fixed $\varepsilon>0,$ putting $t:= \varepsilon\sigma_{ij} \sqrt
{n},$ we obtain
\[
\mathbb P(|Z_{ij,1}| > \varepsilon\sigma_{ij} \sqrt{n}) \leq4de^{-\frac
{\varepsilon\sigma_{ij} \sqrt{n}}{
c_1 d}
}.
\]
Consequently, for the first term in \eqref{perpartes} we have
%
\begin{equation}\label{first_lim}
\lim_{n\rightarrow\infty} \sigma_{ij}^2 {n} P(|Z_{ij,1}|>\varepsilon
\sigma_{ij} \sqrt{n}) \leq\lim_{n\rightarrow\infty}\sigma_{ij}^2 {n}d
e^{-\frac{\varepsilon\sigma_{ij} \sqrt{n}}{
c_1 d}}=0,
\end{equation}
which follows by the sparsity assumption \eqref{sp1} that implies
$d^{\frac{3}{2}} = o({\sqrt{n}}/{\log p}).$
Next considering the limit of the last term in (\ref{perpartes}), we have
\[
\int_{\varepsilon\sigma_{ij} \sqrt{n}}^\infty x P(|Z_{ij,1}|>x)dx \leq
\int_{\varepsilon\sigma_{ij} \sqrt{n}}^\infty4 x d e^{-\frac{x}{c_1 d}}dx.
\]
In the integral, substitute $t:= \frac{x}{\sigma_{ij}\sqrt{n}}$ to obtain
\begin{eqnarray*}
\lim_{n\rightarrow\infty}\int_{\sigma_{ij} \sqrt{n}}^\infty dx
e^{-\frac{x}{d}}dx =
\lim_{n\rightarrow\infty} \int_{1}^\infty d \sigma_{ij}^2 n t e^{-\frac
{\sigma_{ij} \sqrt{n}}{d}t}dt.
\end{eqnarray*}
Again by the restriction on $d$ and the Lebesgue dominated convergence
it then follows that the limit of the integral is $0$.
In conclusion, we get ${S_n}/{s_n}
\rightsquigarrow\mathcal N(0,1)$ for $n\rightarrow\infty$.

\smallskip\noindent(ii) Sub-Gaussian design \ref{subgv}

Under \ref{subgv}, we have a bound $\mathbb P(|Z_{ij,1}| > t) \lesssim
e^{-t/(c_2 K^2)}$
hence similarly as in (i), asymptotic normality follows.
\end{proof}

\begin{lema} \label{subexp}
Let $\Theta^*$ satisfy assumption \ref{eig} and let the random vector
$X\in\mathbb R^p$ satisfy the sub-Gaussianity condition
\ref{subg} with $K=\mathcal O(1)$. Then for $t >c_0$ the random variable
$Z:= {\Theta^*_{ i}}^T X_{} {\Theta^*_{ j}} ^T X_{}-\Theta^*_{ij}$
satisfies the following bound
\[
\mathbb P(|Z_{}| > t) \leq4de^{{-\frac{t}{c_1 d}}},
\]
where $c_0,c_1$ do not depend on $n.$
\end{lema}
\begin{proof}
Since $|{\Theta^*_{ i}}^T X_{} | \leq\|\Theta_{ i}\|_1\max_{k:\Theta
^*_{ki}\not=0} |{X_{}^k}| \leq\|\Theta^*_{ i}\|_2 \sqrt{d} \max
_{k:\Theta^*_{ki}\not=0} |{X_{}^k}|,$
and using the union bound and sub-Gaussianity of $X_{n}^{k}/\sqrt{\Sigma
_{kk}^*}$ by assumption \ref{subg} gives
\begin{eqnarray*}
\mathbb P(|{\Theta^*_{ i}} ^T X_{}| > t)
& \leq&
\mathbb P\left( \max_{k:\Theta^*_{ki}\not=0} |X^k| > \frac{{t}}{\sqrt
{d}\|\Theta^*_{ i}\|_2 }\right) \\
& \leq&
d \max_{k:\Theta^*_{ki}\not=0} \mathbb P\left( |X^k|/\sqrt{\Sigma
_{kk}^*} > \frac{{t}}{\sqrt{d}\|\Theta^*_{ i}\|_2 \sqrt{\Sigma
_{kk}^*}}\right)
\\
&\leq&
2d
\exp\left( {-\frac{{t^2}}{6K^2 d\|\Theta^*_{ i}\|^2_2 \max_{k}{\Sigma
_{kk}^*}} }\right)
.
\end{eqnarray*}
Then for ${\Theta^*_{ i}} ^T X_{} {\Theta^*_{ j}}^T X_{}$ we obtain the bound
\[
\mathbb P(|{\Theta^*_{ i}} ^T X_{} {\Theta^*_{ j}}^T X_{}| > t)
\leq
4d \exp\left( {-\frac{{t}}{6K^2 d\Lambda^2_{\max} (\Theta^*) \max
_{k}{\Sigma_{kk}^*}} }\right)
.
\]
Under \ref{eig},
$\Sigma^*_{kk}$ and $\|\Theta^*_{ i}\|_2 \leq\Lambda_{\max}(\Theta^*)$
are uniformly bounded in $n$.
Thus for $Z_{} = {\Theta^*_{ i}} ^T X_{} {\Theta^*_{ j}}^T X_{}- \Theta
^*_{ij},$ there exist constants $c_0,c_1,c_2$ not depending on $n$ such
that for $t>c_0 > |\Theta^*_{ij}|$
\[
\mathbb P(|Z_{}| > t)
\leq
4de^{{-\frac{t-|\Theta^*_{ij}|}{c_2d}}}\leq4de^{{-\frac{t}{c_1 d}}},
\]
since ${|\Theta^*_{ij}|}/{d}$ is bounded by \ref{eig} and $d\geq1$.
\end{proof}

\begin{proof}[Proof of Lemma \ref{var}]
Since $X\sim\mathcal N(0,\Sigma^*),$ then ${\Theta^*} X \sim\mathcal
N(0,\Theta^*),$
hence
\[
\sigma_{ij}^2 = \text{Var}({\Theta^*_{ i}}^T X{\Theta^*_{ j}}^T X)
= \text{Var}(e_i^T{\Theta^*}^T XX^T{\Theta^*}e_j)=
\text{Var}(e_i^T ZZ^Te_j),
\]
where $Z\sim\mathcal N(0,\Theta^*).$
Thus
\begin{eqnarray*}
\sigma_{ij}^2 &=& \text{Var}(Z^iZ^j)=\mathbb E((Z^i)^2(Z^j)^2)-\mathbb
E(Z^iZ^j)^2
\\
&=&\Theta^*_{ii}\Theta^*_{jj}+2{\Theta^*_{ij}}^2-{\Theta^*_{ij}}^2
= \Theta^*_{ii}\Theta^*_{jj}+{\Theta^*_{ij}}^2.
\end{eqnarray*}
By assumption \ref{eig}, 
${\Theta^*_{ii}\Theta^*_{jj}+{\Theta^*}_{ij}^2} \geq\Lambda^2_{\min
}(\Theta^*) \geq\frac{1}{L^2}>0$, where $L\asymp1$, therefore
$1/\sigma_{ij}=\mathcal O(1)$.

By Lemma \ref{rates} and Lemma \ref{hp} we have
%
\begin{eqnarray}\nonumber
|\hat\sigma_{ij}^2 - \sigma_{ij}^2|
& \leq&
|\hat\Theta_{ii}\hat\Theta_{jj}- \Theta^*_{ii}\Theta^*_{jj}| +
|\hat\Theta_{ij}^2 - {\Theta^*}_{ij}^2 | \\\nonumber
&\leq& |\Delta_{ii}\Delta_{jj} + \Theta^*_{ii}\Delta_{jj}+\Theta
^*_{jj}\Delta_{ii}| \\\nonumber
&&\;+\;\;
|\Delta_{ij}(\Delta_{ij} + 2\Theta^*_{ij})|
\\\label{cons}
&=&\mathcal O_{\mathbb P}\left(
1/\alpha\kappa_{\Gamma^*}\sqrt{\frac{\log p}{n}}
\right)
,
\end{eqnarray}
where we used assumption \ref{eig} and the sparsity assumption \eqref
{sparsity.ravikumar}.
\end{proof}

\subsection*{Concentration for sub-Gaussian design \ref{subgv}}
\label{sec:concentration}
\begin{lema}\label{bilinear}
Let $\alpha,\beta\in\mathbb R^p$ such that $\|\alpha\|_2 \leq M,\|\beta
\|_2\leq M.$
Let $X_k\in\mathbb R^p$ satisfy the sub-Gaussianity assumption \ref{subgv}
with a constant $K>0.$
Then for $m\geq2,$
\[
\mathbb E|\alpha^T X_k X_k^T\beta- \mathbb E \alpha^T X_k X_k^T\beta
|^m / (2 M^2K^2)^{m} \leq
\frac{m!}{2}.
\]
\end{lema}

Consequently, we may apply Bernstein inequality (Lemma 14.9 in \cite{hds}).
\begin{lema}
Let $\alpha,\beta\in\mathbb R^p$ such that $\|\alpha\|_2 \leq M,\|\beta
\|_2\leq M.$
Let $X_k\in\mathbb R^p$ satisfy the sub-Gaussianity assumption \ref{subgv}
with a constant $ K>0.$
For all $t>0$
\[
\mathbb P \left( |\alpha^T \hat\Sigma\beta- \alpha^T \Sigma\beta| /(2M^2K^2)
>
t + \sqrt{2t}
\right) \leq2e^{-nt}.
\]
\end{lema}

\begin{lema} \label{conc}
Assume $\|\alpha_i\|_2 \leq M,\|\beta\|_2\leq M$ for all $i=1,\dots,p$
and \ref{subgv} with~$K$.
For all $t>0$ it holds
\[
\mathbb P
\Biggl( \max_{i=1,\dots,p} |\alpha_i^T(\hat\Sigma-\Sigma) \beta | / (2M^2K^2)
>
t + \sqrt{2t} + \sqrt{\frac{2\log(2p)}{n}} +\frac{\log(2p)}{n}
\Biggr) \leq e^{-nt}.
\]
\end{lema}
\begin{proof}
By Lemma \ref{bilinear} we have
\[
\mathbb E|\alpha^T X_k X_k^T\beta- \mathbb E \alpha^T X_k X_k^T\beta
|^m / (2M^2K^2)^{m} \leq
\frac{m!}{2} .
\]
Lemma 14.13 in \cite{hds} gives the claim.
\end{proof}

\begin{proof}[Proof of Lemma \ref{bilinear}]
By assumption we have $\|\alpha\|_2 \leq M,\|\beta\|_2\leq M.$
Consequently, by the sub-Gaussianity assumption with a constant $K$ we obtain
\[
\mathbb E e^{|X_k^T\alpha|^2/(MK)^2} \leq2
\]
and likewise
\[
\mathbb E e^{|X_k^T \beta|^2/(MK)^2} \leq2.
\]
Next by the inequality $ab\leq a^2/2+b^2/2$ (for any $a,b\in\mathbb R$)
and by the Cauchy-Schwarz inequality we obtain
\begin{eqnarray*}
\mathbb E e^{|\alpha^T X_k X_k^T \beta|/(MK)^2}
&\leq&
\mathbb E e^{| X_k^T \alpha|^2/(MK)^2/2} e^{| X_k^T\beta|^2/(MK)^2/2}
\\
&\leq&
\{\mathbb E e^{| X_k^T \alpha|^2/(MK)^2}\}^{1/2} \{\mathbb Ee^{|
X_k^T\beta|^2/(MK)^2}\}^{1/2}
\\
&\leq& 2.
\end{eqnarray*}
By the Taylor expansion, we have the inequality
\[
1 + \frac{1}{m!}\mathbb E|\alpha^T X_k X_k^T\beta|^m / (MK)^{2m} \leq
\mathbb E e^{|\alpha^T X_k X_k^T \beta|/(MK)^2}
\]
Next it follows
\begin{align*}
\mathbb E|\alpha^T X_k X_k^T\beta- \mathbb E \alpha^T X_k X_k^T\beta
|^m / (MK)^{2m}
& \leq
2^{m-1} \mathbb E|\alpha^T X_k X_k^T\beta|^m / (MK)^{2m} \\
& \leq 2^{m-1} m! (\mathbb Ee^{|\alpha^T X_k X_k^T\beta|/(MK)^{2}}
-1 )\\
&= 2^{m-1}m!
= \frac{m!}{2} 2^{m}.
\end{align*}
And thus
\[
\mathbb E|\alpha^T X_k X_k^T\beta- \mathbb E \alpha^T X_k X_k^T\beta
|^m / (2 M^2K^2)^{m} \leq
\frac{m!}{2} .\qedhere
\]
\end{proof}

\appendix

\section{Tail bounds and rates of convergence}\label{subsec:tails}
The following Lemma from \cite{ravikumar} gives probabilistic bounds
for the event $\mathcal T_n$ if $X$ satisfies
\ref{subg}.

\begin{lema}[\cite{ravikumar}, sub-Gaussian model]\label{hp}
Let $X=(X^1,\dots,X^p)$ be independent, distributed as $X$ with
$\mathbb EX =0,$ $\text{cov}(X) = \Sigma^* = (\Theta^*)^{-1}$
and satisfying \ref{subg} with $K>0$.
Then for
\[
\delta_\tau(n,r) = 8(1+12K^2)\max_i\Sigma^*_{ii} \sqrt{2\frac{\log(4r)}{n}},
\]
and for every $\gamma>2$ and each $n$ such that
$\delta_\tau(n,p^\gamma) < 8(1+12K^2)\max_i\Sigma^*_{ii}$ 
we have
\[
\mathbb P\left(\|\hat\Sigma- \Sigma_{}^*\|_\infty\geq
\delta_\tau(n,p^\gamma)
\right) \leq\frac{1}{p^{\gamma-2}}.
\]
\end{lema}

We restate a result on rates of convergence of the graphical Lasso from
\cite{ravikumar} in Lemma \ref{rates}.
\begin{lema}[Theorem 1, \cite{ravikumar}]\label{rates}
Suppose that $X_1,\dots,X_n\in\mathbb R^p$ are independent and
distributed as $X=(X^1,\dots,X^p)$ with $\mathbb EX = 0,$ $\text
{cov}(X) = \Sigma^*.$ Let $\Theta^*=(\Sigma^*)^{-1}$ exist and satisfy
the irrepresentability condition \ref{ir} with a constant $\alpha\in
(0,1]$. Denote $S^c$ to be the set of indices that correspond to zero
entries of $\Theta^*.$
Let $\hat\Theta_n$ be the solution to the optimization problem {\ref
{alg}} with tuning parameter
$\lambda_n =\frac{8}{\alpha} \delta_n$ for some $\delta_n>0$ and
suppose that the sparsity assumption
\[
d \leq\frac{1}{6(1+8/\alpha) \max\{\kappa_{\Sigma^*}\kappa_{\Gamma
^*},\kappa_{\Sigma^*}^3\kappa_{\Gamma^*}^2 \} \delta_n}
\]
is satisfied.
Then on $\mathcal T_n=\{\|\hat\Sigma-\Sigma\|_\infty\leq\delta_n\}$
it holds
\begin{enumerate}[(a)]
\item
\label{zeros}
$\hat\Theta_{S^c} = \Theta^*_{S^c}$
\item
\label{rate}
$\|\hat\Theta-\Theta^*\|_\infty\leq2(1+8/\alpha)\kappa_{\Gamma^*}
\delta_n.$
\end{enumerate}
\end{lema}

For instance, if the observations $X_1,\dots,X_n$ are sub-Gaussian \ref
{subg} with
$K=\mathcal O(1),$ $\max_{i}\Sigma^*_{ii}=\mathcal O(1)$ then Lemma \ref
{hp} gives
$\|\hat\Sigma-\Sigma^*\|_\infty=\mathcal O_{\mathbb P}(\sqrt{\log p/n}).$

Hence when the observations are sub-Gaussian, Lemma \ref{rates} implies
the following convergence rates for the graphical Lasso estimator.
If the quantities $\kappa_{\Sigma^*}=\mathcal O(1)$, $\kappa_{\Gamma
^*}=\mathcal O(1)$, $1/\alpha=\mathcal O(1)$, the sparsity assumption
$d =o(\sqrt{{n}/{\log p}})$
is satisfied (equivalently $n = \Omega(d^2\log p)$), then for suitably
chosen $\lambda_n\asymp\sqrt{\log p/n}$ we have
\[
\|\hat\Theta-\Theta^*\|_\infty= \mathcal O_{\mathbb P}( \sqrt{{\log p}/{n}}).
\]

\end{document}